\documentclass[a4paper]{amsart}
\usepackage[dvipdfmx]{graphicx}
\usepackage{color}
\usepackage{amsmath,amssymb}
\usepackage{amsthm}
\usepackage{here}
\usepackage{comment}
\usepackage{graphicx}
\newtheorem{theorem}{Theorem}[section]

\newtheorem{corollary}[theorem]{Corollary}
\newtheorem{lemma}[theorem]{Lemma}
\newtheorem{proposition}[theorem]{Proposition}
\theoremstyle{definition}
\newtheorem{definition}[theorem]{Definition}
\newtheorem{example}[theorem]{Example}
\theoremstyle{remark}

\newenvironment{psmatrix}{\left(\begin{smallmatrix}}{\end{smallmatrix}\right)}

\newcommand{\BB}{\mathbb{B}}
\newcommand{\CC}{\mathbb{C}}
\newcommand{\DD}{\mathbb{D}}
\newcommand{\FF}{\mathbb{F}}
\newcommand{\RR}{\mathbb{R}}
\newcommand{\TT}{\mathbb{T}}
\newcommand{\ZZ}{\mathbb{Z}}
\newcommand{\EE}{\mathbb{E}}

\newcommand{\SO}{\operatorname{SO}}

\newcommand{\hG}{G}
\newcommand{\bG}{\overline{G}}
\newcommand{\bX}{Y}

\renewcommand{\o}{\bm{0}}

\usepackage[whole]{bxcjkjatype}

\numberwithin{equation}{section}

\usepackage{bm}

\usepackage{color}
\definecolor{darkgreen}{cmyk}{1,0,1,.2}
\definecolor{darkorchid}{rgb}{0.6, 0.2, 0.8}
\definecolor{persimmon}{rgb}{0.93, 0.35, 0.0}
\long\def\red#1{\textcolor {red}{#1}} 
\long\def\blue#1{\textcolor {blue}{#1}}
\long\def\green#1{\textcolor {darkgreen}{#1}}

\newdimen\theight
\def\TeXref#1{%
             \leavevmode\vadjust{\setbox0=\hbox{{\tt
                     \quad\quad  {\small \textrm #1}}}%
             \theight=\ht0
             \advance\theight by \lineskip
             \kern -\theight \vbox to
             \theight{\rightline{\rlap{\box0}}%
             \vss}%
             }}%
\newcommand{\fnote}[1]{\TeXref{*}{\tiny{#1}}}

\long\def\bnote#1{\blue{\fnote{#1}}}
\long\def\gnote#1{\green{\fnote{#1}}}

\begin{document}
\title{Generalized Bishop frames on curves on $\EE^{4}$}
\author{Subaru Nomoto}
\author{Hiraku Nozawa}
\keywords{Space curve, Frenet frame, Bishop frame, rotation minimizing vector field}
\subjclass[2020]{}
\date{}


\address{Subaru Nomoto, Graduate School of Science and Engineering, Ritsumeikan University, Nojihigashi 1-1-1, Kusatsu, Shiga, 525-8577, Japan}
\email{gr0317ep@ed.ritsumei.ac.jp}

\address{Hiraku Nozawa, Department of Mathematical Sciences, Colleges of Science and Engineering, Ritsumeikan University, Nojihigashi 1-1-1, Kusatsu, Shiga, 525-8577, Japan}
\thanks{The second author is supported by JSPS KAKENHI Grant numbers 17K14195 and 20K03620}
\email{hnozawa@fc.ritsumei.ac.jp}

  \maketitle 
\begin{abstract}
    We introduce and study generalized Bishop frames on regular curves, which are generalizations of the Frenet and Bishop frames for regular curves on higher dimensional spaces. There are four types of generalized Bishop frames on regular curves on $\EE^{4}$ up to the change of the order of vectors fixing the first one which is the tangent vector. One of these four types of frames is a Bishop frame, and by a result of Bishop, every regular curve admits such a frame. We show that if a regular curve $\gamma$ on $\EE^{4}$ admits a Frenet frame, then $\gamma$ admits all four types of generalized Bishop frames. We also show that if the derivative of the tangent vector of a regular curve is nowhere vanishing, then the curve admits all three types of generalized Bishop frames except a frame of type F, which is related to the Frenet frame.
 \end{abstract}
 
\section{Introduction}

The Frenet frame is fundamental to study space curves, which allows us to classify the congruence classes of curves in terms of its torsion and curvature. The Bishop frame (or rotation minimizing frame) is a different type of frames on curves introduced by Bishop \cite{b}. It is an advantage of the Bishop frame that every regular curve admits the Bishop frame by a theorem of Bishop \cite{b}, while some regular curves do not admit Frenet frame.
Bishop frames are used in the study of curves, see \cite{KT2,dS,E,GN}. For the application of Bishop frames in differential geometry, see \cite{KT,BG,dSdS,MR}. The open problems on Bishop frames are listed in \cite{Fa}. The Bishop frame was used in other domains including computer graphics and engineering, see \cite{W,FL}.

In this article, we study relations between frames on curves on $\mathbb{E}^{4}$ obtained by generalizing the Frenet and Bishop frames. In order to motivate us, let us review the idea of Bishop to introduce Bishop frames. For a regular curve $\gamma : I \to \EE^{3}$ defined on an open interval $I$ with arc-length parametrization, a Bishop frame is a frame of $\gamma^{*}T\mathbb{E}^3$ 
of the form $\{\mathbb{T}, \mathbb{B}_{1}, \mathbb{B}_{2}\}$ where $\mathbb{T}$ is the tangent vector and satisfies the following differential equation: 
\[\begin{pmatrix}
\mathbb{T}^{\prime}\\
\mathbb{B}_{1}^{\prime}\\
\mathbb{B}_{2}^{\prime}
\end{pmatrix}
=\begin{pmatrix}
0&b_{1}&b_{2}\\
-b_{1}&0&0\\
-b_{2}&0&0
\end{pmatrix}
\begin{pmatrix}
\mathbb{T}\\
\mathbb{B}_{1}\\
\mathbb{B}_{2}
\end{pmatrix}
.
\]
Throughout this article, regular curves are parametrized by arc-length parameters unless otherwise stated. We will consider only frames on regular curves whose first vector is the tangent vector. We regard an orthonormal frame on a curve $I \to \EE^n$ as a matrix valued function $\ZZ : I \to O(n)$ such that the frame consists of the row vectors of $\ZZ$. For a frame on a regular curve, we will call the matrix valued function $X$ such that $\ZZ' = X\ZZ$ the coefficient matrix of the frame.
In order to introduce Bishop frames, Bishop considered three types of frames, whose coefficient matrices have one of the following forms for some functions $x_1$ and $x_2$:

\begin{minipage}{0.3\hsize}
\[\left(\begin{matrix}
0&x_{1}&0\\
-x_{1}&0&x_{2}\\
0&-x_{2}&0
\end{matrix}
\right), \]
\hspace*{45pt}(1)
\end{minipage}
\begin{minipage}{0.3\hsize}
\[\left(\begin{matrix}
0&x_{1}&x_{2}\\
-x_{1}&0&0\\
-x_{2}&0&0
\end{matrix}
\right), \]
\hspace*{45pt}(2)
\end{minipage}
\begin{minipage}{0.3\hsize}
\[\left(\begin{matrix}
0&0&x_{1}\\
0&0&x_{2}\\
-x_{1}&-x_{2}&0
\end{matrix}
\right). \]
\hspace*{45pt}(3)
\end{minipage}

A frame with the coefficient matrix (2) is exactly the Bishop frame. A frame with coefficient matrix of the form (1) is closely related to the Frenet frame; it is the Frenet frame if $x_1$ is positive.
Bishop remarked that if a frame $\{\TT, \ZZ_1, \ZZ_2\}$ has a coefficient matrix of the form (3), then the coefficient matrix of $\{\TT, \ZZ_2, \ZZ_1\}$ is of the form (1). Therefore, there are essentially two types among this kinds of frames on curves on $\EE^{3}$, which are the Frenet frame and Bishop frame. Generalizations of these frames on curves on general Riemannian manifolds are known 
\cite{E}
: let $\gamma : I \to \EE^{4}$ be a regular curve on a Riemannian manifold. An orthonormal frame $\TT, \ZZ_1, \ZZ_2, \ZZ_3$ on $\gamma$ is the Frenet frame or Bishop frame if its coefficient matrix $X$ is of the form 
\begin{align*}
\begin{pmatrix}
\phantom{-}0&\phantom{-}b_{1}&\phantom{-}b_{2}&\phantom{-}b_{3}\\
-b_{1}&\phantom{-}0&\phantom{-}0&\phantom{-}0\\
-b_{2}&\phantom{-}0&\phantom{-}0&\phantom{-}0\\
-b_{3}&\phantom{-}0&\phantom{-}0&\phantom{-}0
\end{pmatrix}&&&\hspace{15pt}\text{or}\hspace{-5pt}&&&
\begin{pmatrix}
\phantom{-}0&\phantom{-}f_{1}&\phantom{-}0&\phantom{-}0\\
-f_{1}&\phantom{-}0&\phantom{-}f_{2}&\phantom{-}0\\
\phantom{-}0&-f_{2}&\phantom{-}0&\phantom{-}f_{3}\\
\phantom{-}0&\phantom{-}0&-f_{3}&\phantom{-}0
\end{pmatrix}, 
\end{align*}
respectively, where $b_1, b_2, b_3,f_3$ are functions on $I$ and $f_1, f_2$ are positive functions on $I$. Generalizing the case of space curves, we give the following definition:
\begin{definition}
We call an orthonormal frame on a curve on $\EE^{4}$ a \emph{generalized Bishop frame} if its coefficient matrix has at most three nonzero entries above the main diagonal. 
\end{definition}
Except degenerate ones which has a zero column vector, there are 16 kinds of such frames (see Section \ref{sec:pre}). We can see that there are the following 4 equivalence classes of these 16 frames on curves up to the change of the order of vectors fixing the first one which is the tangent vector: If a frame has a coefficient matrix of the respective form for some functions $x_1,x_2,x_3$ up to the change of the order of vectors fixing the first one, we call it a \emph{generalized Bishop frame} of type B, C, D and F, respectively:
\begin{align*}
\begin{pmatrix}
\phantom{-}0&\phantom{-}x_{1}&\phantom{-}x_{2}&\phantom{-}x_{3}\\
-x_{1}&\phantom{-}0&\phantom{-}0&\phantom{-}0\\
-x_{2}&\phantom{-}0&\phantom{-}0&\phantom{-}0\\
-x_{3}&\phantom{-}0&\phantom{-}0&\phantom{-}0
\end{pmatrix}
, && 
\begin{pmatrix}
\phantom{-}0&\phantom{-}x_{1}&\phantom{-}x_{2}&\phantom{-}0\\
-x_{1}&\phantom{-}0&\phantom{-}0&\phantom{-}x_{3}\\
-x_{2}&\phantom{-}0&\phantom{-}0&\phantom{-}0\\
\phantom{-}0&-x_{3}&\phantom{-}0&\phantom{-}0
\end{pmatrix}, \\
\text{Type B} \hspace{40pt}&& \text{Type C}\hspace{40pt} 
\end{align*}
\begin{align*}
\begin{pmatrix}
\phantom{-}0&\phantom{-}x_{1}&\phantom{-}0&\phantom{-}0\\
-x_{1}&\phantom{-}0&\phantom{-}x_{2}&\phantom{-}x_{3}\\
\phantom{-}0&-x_{2}&\phantom{-}0&\phantom{-}0\\
\phantom{-}0&-x_{3}&\phantom{-}0&\phantom{-}0
\end{pmatrix}, &&
\begin{pmatrix}
\phantom{-}0&\phantom{-}x_{1}&\phantom{-}0&\phantom{-}0\\
-x_{1}&\phantom{-}0&\phantom{-}x_{2}&\phantom{-}0\\
\phantom{-}0&-x_{2}&\phantom{-}0&\phantom{-}x_{3}\\
\phantom{-}0&\phantom{-}0&-x_{3}&\phantom{-}0
\end{pmatrix}. \\
\text{Type D}\hspace{40pt} && \text{Type F}\hspace{40pt} 
\end{align*}
A frame of type F is closely related to the Frenet frame; it is Frenet frame if both $x_1$ and $x_2$ are positive. A frame of type B is a Bishop frame. By Bishop's theorem, every regular curve admits a Bishop frame, but there are well known examples of regular curves which do not admit the Frenet frame. 

In general, the relation between these four types of frames are not clear. But the following result shows that not all regular curves admit generalized Bishop frames of type C and D. Therefore it is a geometric property for regular curves whether they admit these frames or not.

\begin{theorem}\label{thm:noCD}
There is a regular curve which admits a generalized Bishop frame of type C but does not admit a frame of type D.
There is a regular curve which does not admit a frame of type C.
\end{theorem}


In this article, we focus on the frames on the following types of curves.
\begin{definition}
A curve is said to be \emph{$2$-regular} if both the tangent vector and its derivative are nowhere vanishing. 
\end{definition}

The main results of this article are the following:
\begin{theorem}\label{thm:CD}
Every $2$-regular curve on $\EE^{4}$ admits a frame of type C and D, respectively.
\end{theorem}
Note that every regular curve admits a Bishop frame by Bishop's theorem. The following result implies that the frame
 of type F on $2$-regular curves is distinguished from other types.
 
\begin{theorem}\label{thm:noF}
There exists a $2$-regular curve on $\EE^{4}$ which does not admit a frame of type F.
\end{theorem}

A basic tool of the proof is the differential equation of transformations between two frames (Lemma \ref{lem:1}). Let $\gamma$ be an arc-length parametrized regular curve with a Bishop frame $\BB$. 
Assume that $\gamma$ admits a frame $\FF$ of type F. Then we observe that  the transformation $\hG = \FF \BB^{-1}$ from  $\BB$ to $\FF$ satisfies the following differential equation (see Lemma \ref{lem:1} below):
\begin{equation}\label{eq:trans}
\hG' = X_{\FF} \hG - \hG X_{\BB},  
\end{equation}
where $X_{\BB}$ and $X_{\FF}$ are the coefficient matrix of $\BB$ and $\FF$, respectively. We can see that, conversely, given $X_{\BB}$, if the equation \eqref{eq:trans} has a solution $\hG$ for some $X_{\FF}$ of the form of coefficient matrix of a frame of type F, then $\hG\BB$ is a frame of type F (see Lemma \ref{lem:1} below). In this way, more generally for frames of other types, we can translate the existence problem of a frame of given type on a curve to the existence of a solution of a differential equation such that a part of the coefficients, the entries of the coefficient matrix of the target frame, are undetermined. The main results are proved by showing the existence of solutions of the differential equations similar to \eqref{eq:trans} under a suitable choice of the undetermined coefficients. 


\section{Basic properties of frames}\label{sec:pre}



There are 16 kinds of generalized Bishop frames on regular curves on $\EE^{4}$, excluding frames whose coefficient matrix has a zero column vector. We will study the relation between these frames. It is easy to see that they are classified into four types up to the action of the symmetric group $\mathfrak{S}_{3}$ of order $3$ which swaps the second, third and fourth vector of the frame as follows:

\[\begin{psmatrix}
\phantom{-}0&\phantom{-}\blacksquare&\phantom{-}\blacksquare&\phantom{-}\blacksquare\\
-\blacksquare&\phantom{-}0&\phantom{-}0&\phantom{-}0\\
-\blacksquare&\phantom{-}0&\phantom{-}0&\phantom{-}0\\
-\blacksquare&\phantom{-}0&\phantom{-}0&\phantom{-}0
\end{psmatrix}\]
\begin{center}type B\end{center}

\begin{align*}
\begin{psmatrix}
\phantom{-}0&\phantom{-}\blacksquare&\phantom{-}\blacksquare&\phantom{-}0\\
-\blacksquare&\phantom{-}0&\phantom{-}0&\phantom{-}\blacksquare\\
-\blacksquare&\phantom{-}0&\phantom{-}0&\phantom{-}0\\
\phantom{-}0&-\blacksquare&\phantom{-}0&\phantom{-}0
\end{psmatrix}
\begin{psmatrix}
\phantom{-}0&\phantom{-}\blacksquare&\phantom{-}\blacksquare&\phantom{-}0\\
-\blacksquare&\phantom{-}0&\phantom{-}0&\phantom{-}0\\
-\blacksquare&\phantom{-}0&\phantom{-}0&\phantom{-}\blacksquare\\
\phantom{-}0&\phantom{-}0&-\blacksquare&\phantom{-}0
\end{psmatrix}
\begin{psmatrix}
\phantom{-}0&\phantom{-}\blacksquare&\phantom{-}0&\phantom{-}\blacksquare\\
-\blacksquare&\phantom{-}0&\phantom{-}0&\phantom{-}0\\
\phantom{-}0&\phantom{-}0&\phantom{-}0&\phantom{-}\blacksquare\\
-\blacksquare&\phantom{-}0&-\blacksquare&\phantom{-}0
\end{psmatrix}\\
\begin{psmatrix}
\phantom{-}0&\phantom{-}\blacksquare&\phantom{-}0&\phantom{-}\blacksquare\\
-\blacksquare&\phantom{-}0&\phantom{-}\blacksquare&\phantom{-}0\\
\phantom{-}0&-\blacksquare&\phantom{-}0&\phantom{-}0\\
-\blacksquare&\phantom{-}0&\phantom{-}0&\phantom{-}0
\end{psmatrix}
\begin{psmatrix}
\phantom{-}0&\phantom{-}0&\phantom{-}\blacksquare&\phantom{-}\blacksquare\\
\phantom{-}0&\phantom{-}0&\phantom{-}\blacksquare&\phantom{-}0\\
-\blacksquare&-\blacksquare&\phantom{-}0&\phantom{-}0\\
-\blacksquare&\phantom{-}0&\phantom{-}0&\phantom{-}0
\end{psmatrix}\begin{psmatrix}
\phantom{-}0&\phantom{-}0&\phantom{-}\blacksquare&\phantom{-}\blacksquare\\
\phantom{-}0&\phantom{-}0&\phantom{-}0&\phantom{-}\blacksquare\\
-\blacksquare&\phantom{-}0&\phantom{-}0&\phantom{-}0\\
-\blacksquare&-\blacksquare&\phantom{-}0&\phantom{-}0
\end{psmatrix}
\end{align*}
\begin{center}type C\end{center}

\[\begin{psmatrix}
\phantom{-}0&\phantom{-}\blacksquare&\phantom{-}0&\phantom{-}0\\
-\blacksquare&\phantom{-}0&\phantom{-}\blacksquare&\phantom{-}\blacksquare\\
\phantom{-}0&-\blacksquare&\phantom{-}0&\phantom{-}0\\
\phantom{-}0&-\blacksquare&\phantom{-}0&\phantom{-}0
\end{psmatrix}
\begin{psmatrix}
\phantom{-}0&\phantom{-}0&\phantom{-}\blacksquare&\phantom{-}0\\
\phantom{-}0&\phantom{-}0&\phantom{-}\blacksquare&\phantom{-}0\\
-\blacksquare&-\blacksquare&\phantom{-}0&\phantom{-}\blacksquare\\
\phantom{-}0&\phantom{-}0&-\blacksquare&\phantom{-}0
\end{psmatrix}
\begin{psmatrix}
\phantom{-}0&\phantom{-}0&\phantom{-}0&\phantom{-}\blacksquare\\
\phantom{-}0&\phantom{-}0&\phantom{-}0&\phantom{-}\blacksquare\\
\phantom{-}0&\phantom{-}0&\phantom{-}0&\phantom{-}\blacksquare\\
-\blacksquare&-\blacksquare&-\blacksquare&\phantom{-}0\\
\end{psmatrix}
\]
\begin{center}type D\end{center}

\begin{align*}
\begin{psmatrix}
\phantom{-}0&\phantom{-}\blacksquare&\phantom{-}0&\phantom{-}0\\
-\blacksquare&\phantom{-}0&\phantom{-}\blacksquare&\phantom{-}0\\
\phantom{-}0&-\blacksquare&\phantom{-}0&\phantom{-}\blacksquare\\
\phantom{-}0&\phantom{-}0&-\blacksquare&\phantom{-}0
\end{psmatrix}
\begin{psmatrix}
\phantom{-}0&\phantom{-}0&\phantom{-}0&\phantom{-}\blacksquare\\
\phantom{-}0&\phantom{-}0&\phantom{-}\blacksquare&\phantom{-}\blacksquare\\
\phantom{-}0&-\blacksquare&\phantom{-}0&\phantom{-}0\\
-\blacksquare&-\blacksquare&\phantom{-}0&\phantom{-}0
\end{psmatrix}
\begin{psmatrix}
\phantom{-}0&\phantom{-}0&\phantom{-}0&\phantom{-}\blacksquare\\
\phantom{-}0&\phantom{-}0&\phantom{-}\blacksquare&\phantom{-}0\\
\phantom{-}0&-\blacksquare&\phantom{-}0&\phantom{-}\blacksquare\\
-\blacksquare&\phantom{-}0&-\blacksquare&\phantom{-}0
\end{psmatrix}\\
\begin{psmatrix}
\phantom{-}0&\phantom{-}\blacksquare&\phantom{-}0&\phantom{-}0\\
-\blacksquare&\phantom{-}0&\phantom{-}0&\phantom{-}\blacksquare\\
\phantom{-}0&\phantom{-}0&\phantom{-}0&\phantom{-}\blacksquare\\
\phantom{-}0&-\blacksquare&-\blacksquare&\phantom{-}0
\end{psmatrix}
\begin{psmatrix}
\phantom{-}0&\phantom{-}0&\phantom{-}\blacksquare&\phantom{-}0\\
\phantom{-}0&\phantom{-}0&\phantom{-}\blacksquare&\phantom{-}\blacksquare\\
-\blacksquare&-\blacksquare&\phantom{-}0&\phantom{-}0\\
\phantom{-}0&-\blacksquare&\phantom{-}0&\phantom{-}0
\end{psmatrix}
\begin{psmatrix}
\phantom{-}0&\phantom{-}0&\phantom{-}\blacksquare&\phantom{-}0\\
\phantom{-}0&\phantom{-}0&\phantom{-}0&\phantom{-}\blacksquare\\
-\blacksquare&\phantom{-}0&\phantom{-}0&\phantom{-}\blacksquare\\
\phantom{-}0&-\blacksquare&-\blacksquare&\phantom{-}0
\end{psmatrix}
\end{align*}
\begin{center}type F\end{center}

A frame of type B is nothing but a Bishop frame. 
Clearly a regular curve admits one of generalized Bishop frames if and only if it admits all generalized Bishop frames of the same type. 

We can characterize the curves which admit a frame of type D as follows.

\begin{proposition}\label{proposition:1}
For a regular curve $\gamma$ on $\EE^{4}$, the following are equivalent:
\begin{enumerate}
    \item $\gamma$ admits a frame of type D.
    \item There exists a smooth unit normal vector field $\mathbb{D}_{1}$ and a smooth function $d_1$ such that  $\mathbb{T}^{\prime}=d_{1}\mathbb{D}_{1}$.
\end{enumerate}
\end{proposition}

\begin{proof}
Clearly (1) implies (2). Let us show (1) assuming (2). Let $\gamma$ be a regular curve which satisfies the condition (2). Let $\mathbb{D}_{1}$ be a unit normal vector field on $\gamma$ such that $\mathbb{T}^{\prime}=d_{1}\mathbb{D}_{1}$. Let us consider a regular curve $\delta$ whose tangent vector is $\DD_1$. By a result of Bishop, $\delta$ admits a  Bishop frame of the form $\DD_1, \TT, \DD_2,\DD_3$. 
Then we have
\[\left(\begin{matrix}
\mathbb{D}_{1}^{\prime}\\
\mathbb{T}^{\prime}\\
\mathbb{D}_{2}^{\prime}\\
\mathbb{D}_{3}^{\prime}
\end{matrix}\right)=\left(\begin{matrix}
0&x_{1}&x_2&x_3\\
-x_{1}&0&0&0\\
-x_{2}&0&0&0\\
-x_{3}&0&0&0
\end{matrix}\right)\left(\begin{matrix}
\mathbb{D}_{1}\\
\mathbb{T}\\
\mathbb{D}_{2}\\
\mathbb{D}_{3}
\end{matrix}\right)\]
for some $x_1, x_2, x_3$, which implies that $\TT, \DD_1, \DD_2,\DD_3$ is a frame of type D.
\end{proof}

With this proposition, it is easy to construct an example of regular curve which does not admit a frame of type D, which shows the first half of Theorem \ref{thm:noCD}: 

\begin{example}\label{ex:noD}
Let $\gamma$ be a curve on $\EE^{4}$ defined by
\begin{equation}\label{eq:gamma}
\gamma(t)=
\begin{cases}(t, e^{-\frac{1}{t}},0,0) & t>0\\
(0, 0, 0, 0) & t=0\\
(t,0, e^{\frac{1}{t}},0) & t<0.
\end{cases}
\end{equation}
Note that the parameter of $\gamma$ is not an arc-length.
 Then, it is easy to see that $\gamma''/|\gamma''|$ is not continuous at $t=0$. Therefore, $\gamma$ does not admit a frame of type D. It is easy to see that this $\gamma$ admits a frame of type C, and hence it shows the first half of Theorem \ref{thm:noCD}.
\end{example}

The later half of Theorem \ref{thm:noCD} is proved in Section \ref{sec:ex}.

If a regular curve admits a Frenet frame, then the normalization of the differential of the tangent vector satisfies the condition for $\mathbb{D}_{1}$ in (2) of the last proposition. Thus we get the following:

\begin{corollary}
If a regular curve admits a frame of type F, then it admits a frame of type D.
\end{corollary}

\begin{corollary}
Every $2$-regular curve admits a frame of type D.
\end{corollary}

It is more involved to characterize curves which admits a frame of type C.

\begin{proposition}
A regular curve $\gamma$ admits a frame of type C if and only if there exists a unit normal vector field $\ZZ_{0}$ on $\gamma$ such that a curve whose tangent vector is $\ZZ_{0}$ admits a frame of type F of the form $\ZZ_{0}, \ZZ_{1}, \TT, \ZZ_{3}$, or of the form$\ZZ_{0}, \TT, \ZZ_{1}, \ZZ_{3}$
whose coefficient matrix is of the form
\begin{equation}\label{eq:frenet}
\begin{pmatrix}
\phantom{-}0&\phantom{-}f_{1}&\phantom{-}0&\phantom{-}0\\
-f_{1}&\phantom{-}0&\phantom{-}f_{2}&\phantom{-}0\\
\phantom{-}0&-f_{2}&\phantom{-}0&\phantom{-}f_{3}\\
\phantom{-}0&\phantom{-}0&-f_{3}&\phantom{-}0
\end{pmatrix}.
\end{equation}
\end{proposition}

\begin{proof}
It is easy to see that the coefficient matrix of the frame $\ZZ_{0}, \ZZ_{1}, \TT, \ZZ_{3}$ is of the form \eqref{eq:frenet} if and only if we have
\[
\begin{pmatrix}
\mathbb{T}^{\prime}\\
\ZZ_{1}^{\prime}\\
\ZZ_{0}^{\prime}\\
\ZZ_{3}^{\prime}
\end{pmatrix}=\begin{pmatrix}
0&-f_{2}&0&f_3\\
f_{2}&0&-f_{1}&0\\
0&f_{1}&0&0\\
-f_3&0&0&0
\end{pmatrix}\begin{pmatrix}
\mathbb{T}\\
\ZZ_{1}\\
\ZZ_{0}\\
\ZZ_{3}
\end{pmatrix}.
\]
The latter hold if and only if $\TT, \ZZ_{1}, \ZZ_{0}, \ZZ_{3}$ of type C.
\end{proof}

By the last proposition, a curve contained in a hyperplane of $\EE^{4}$ admits a frame of type C. In particular, the curve $\gamma$ in \eqref{eq:gamma} admits a frame of type C. Therefore it is an example of a regular curve which admits a frame of type C but does not admit a frame of type D.

Typical examples of curves with generalized Bishop frame are curves with frames with constant coefficient matrices. For example, let $X= \begin{pmatrix}
0&2&1&0\\
-2&0&0&1\\
-1&0&0&0\\
0&-1&0&0
\end{pmatrix}$. Then $\exp (sX)$ is a generalized Bishop frame of type C on a curve $\gamma$ whose tangent vector is the first column vector of $\exp (sX)$, where

{\fontsize{7pt}{7pt}\selectfont \begin{multline*}\gamma(s)=\\
\left(\frac{\sin \sqrt{2} s \cos
   s}{\sqrt{2}},\frac{\sin s \sin \sqrt{2}
   s}{\sqrt{2}},\frac{-\sin s \sin
   \sqrt{2} s-\sqrt{2} \cos s \cos
   \sqrt{2} s}{\sqrt{2}},\frac{\sin
   \sqrt{2} s \cos s-\sqrt{2} \sin s
   \cos \sqrt{2} s}{\sqrt{2}}\right),\end{multline*}}
Let us present an example of a $2$-regular curve which does not admit a frame of type F, which shows Theorem \ref{thm:noF}:

\begin{example}
Consider a curve $\gamma$ on $\EE^{4}$ whose tangent vector field is
\begin{equation}
\TT(s)=\begin{cases}
\displaystyle \left(\frac{2s}{s^{2}+e^{-\frac{2}{s}}+1}, \, \frac{2e^{-\frac{1}{s}}}{s^{2}+e^{-\frac{2}{s}}+1}, \, 0, \, \frac{s^{2}+e^{-\frac{2}{s}}-1}{s^{2}+e^{-\frac{2}{s}}+1}\right) & s> 0\\
(0, 0, 0, -1) & s=0\\
\displaystyle \left(\frac{2s}{s^{2}+e^{\frac{2}{s}}+1},\, 0, \, \frac{2e^{\frac{1}{s}}}{s^{2}+e^{\frac{2}{s}}+1}, \, \frac{s^{2}+e^{\frac{2}{s}}-1}{s^{2}+e^{\frac{2}{s}}+1}\right) & s< 0.\end{cases}
\end{equation}
It is easy to see that $\gamma$ is $2$-regular.  

Assume that $\gamma$ admits a generalized Bishop frame $\TT, \FF_1, \FF_2, \FF_3$ of type $F$. Since $\gamma (\RR_{\geq 0})$ is contained in a hyperplane $H_{+}=\{\, (x,y,z,w) \in \EE^{4} \mid z=0 \}$ and the curvature and the torsion of the Frenet frame of $\gamma$ does not vanish on $\RR_{>0}$,
it follows that $\TT$, $\FF_1$, $\FF_2$ spans $H_{+}$ at each point on $\gamma$.
 Similarly, since $\gamma (\RR_{\leq 0})$ is contained in a hyperplane $H_{-}=\{\, (x,y,z,w) \in \EE^{4} \mid y=0 \}$, by using the Frenet frame of $\gamma$, we can see that $\TT$, $\FF_1$, $\FF_2$ spans $H_{-}$ at each point on $\gamma$. Thus, at $t=0$, $\TT, \FF_1, \FF_2$ must span $H_{+} \cap H_{-}$, which is a contradiction. Therefore, $\gamma$ does not admit a frame of type F.
\end{example}

The tangent vector $\TT$ of this example is similar to the curve in Example \ref{eq:gamma}, which can be divided into two curves, each of which is contained in a hyperplane. 

\section{Transformations between frames}

As in the following lemma, by considering transformations between frames, we can reduce the problem of the existence of a frame of given coefficient matrix to the existence of a solution of a differential equation given by the entries of coeffient matrices. 
\begin{lemma}\label{lem:1}
Let $\gamma : I \to \EE^{4}$ be a regular curve. Let $\ZZ_0$ be a frame on $\gamma$ such that $\ZZ'_0 = X_0 \ZZ_0$, and a function $X_1 : I \to \mathfrak{o}(4)$. Consider a frame $\ZZ_1$ on $\gamma$ such that 
\begin{equation}\label{eq:X1}
\ZZ'_1 = X_1 \ZZ_1. 
\end{equation}
Then the transformation $\hG : I \to O (4)$ from $\ZZ_0$ to $\ZZ_1$ given by $\hG = \ZZ_1 \ZZ_0^{-1}$ satisfies the differential equation
\begin{equation}
\label{eq:a}{\hG}^{\prime}=X_{1}{\hG}-{\hG}X_{0}.
\end{equation}
On the contrary, if a function $\hG : I \to O (4)$ satisfies the differential equation \eqref{eq:a}, then the frame $\ZZ_1$ given by $\ZZ_1=\hG \ZZ_0$ satisfies the differential equation \eqref{eq:X1}.
\end{lemma}

\begin{proof}
In order to prove the first part, assume that $\ZZ_1$ satisfies \eqref{eq:X1}. By differentiating both sides of $\ZZ_0\ZZ_0^{-1} = E$, we get $(\ZZ_0^{-1})' = - \ZZ_0^{-1} \ZZ'_0\ZZ_0^{-1}$. Then we have
\begin{multline*}
{\hG}^{\prime}=(\ZZ_1\ZZ_0^{-1})^{\prime}  =\ZZ^{\prime}_{1}\ZZ_0^{-1}+\ZZ_1(\ZZ_0^{-1})^{\prime}  \\ =X_{1}\ZZ_1\ZZ_0^{-1}+\ZZ_1(-\ZZ_0^{-1}\ZZ_0^{\prime}\ZZ_0^{-1}) =X_{1}{\hG}-{\hG}X_{0},
\end{multline*}
from which follows the differential equation \eqref{eq:a}. For the proof of the second part, assume that ${\hG}$ satisfies \eqref{eq:a} and consider $\ZZ_1 = {\hG}\ZZ_0$. Then we have
\[
\ZZ'_{1} = ({\hG}\ZZ_0)^{\prime} = {\hG}^{\prime}\ZZ_0 + {\hG}\ZZ^{\prime}_{0} =(X_{1}{\hG}-{\hG}X_{0})\ZZ_0 +{\hG}X_{0}\ZZ_0 =X_{1}\ZZ_1,
\]
which is \eqref{eq:X1}.
\end{proof}

By using the lemma, we can relate curvatures of frames of different types. The curvature of a frame of type D is related to the curvature of a frame of type F in a simple way.

\begin{proposition}\label{prop:1}
Let
\begin{equation}\label{eq:FDc}
X_{\FF}=\begin{pmatrix}
0&f_{1}&0&0\\
-f_{1}&0&f_{2}&0\\
0&-f_{2}&0&f_{3}\\
0&0&-f_{3}&0
\end{pmatrix}, \quad\quad 
X_{\DD}=\begin{pmatrix}
0&d_{1}&0&0\\
-d_{1}&0&d_{2}&d_{3}\\
0&-d_{2}&0&0\\
0&-d_{3}&0&0
\end{pmatrix},
\end{equation}
where $f_1, f_2, f_3, d_1, d_2, d_3$ are functions on an interval $I$. Let $\gamma : I \to \EE^{4}$ be a $2$-regular curve $\gamma : I \to \EE^{4}$ admits a generalized Bishop frame $\FF$ of type $F$ with coefficient matrix $X_{\FF}$, then $\gamma$ admits a generalized Bishop frame $\DD$ of type D with coefficient matrix $X_{\DD}$, where we have
\begin{equation}\label{eq:FD}
d_{1}=\epsilon f_{1}, \quad d_{2}=\epsilon f_{2} \cos \left(\int f_{3} ds\right), \quad d_{3}= -\kappa f_{2}\sin \left(\int f_{3}ds\right)
\end{equation}
for some $\epsilon, \kappa \in \{1,-1\}$. 
\end{proposition}

\begin{proof}
By Lemma \ref{lem:1}, it suffices to show that the solution ${\hG} : I \to O(4)$ of a differential equation 
\begin{equation}\label{eq:FDt}
    {\hG}^{\prime}=X_{\DD}{\hG}-{\hG}X_{\FF}
\end{equation}
exists if and only if $d_1, d_2, d_3$ satisfy the relation \eqref{eq:FD}.
Letting 
 \[
\hG=
\begin{pmatrix}
1&\o\\
{}^{t}\o&\bG
\end{pmatrix}, \quad
X_{\FF} = \begin{pmatrix}
0&{\bm f}\\
-^{t}{\bm f}&\bX_{\FF}
\end{pmatrix}, \quad  
X_{\DD}=\begin{pmatrix}
0&{\bm d}\\
-^{t}{\bm d}&\bX_{\DD}
\end{pmatrix}
 \]
the equation \eqref{eq:FDt} is expressed as
\begin{align*}
\begin{pmatrix}
0&\o\\
{}^{t}\o&\bG^{\prime}
\end{pmatrix}
&=
\begin{pmatrix}
0&{\bm d}\\
-^{t}{\bm d}&\bX_{\DD}
\end{pmatrix}
\begin{pmatrix}
1&\o\\
{}^{t}\o&\bG
\end{pmatrix}
-
\begin{pmatrix}
1&\o\\
{}^{t}\o&\bG
\end{pmatrix}
\begin{pmatrix}
0&{\bm f}\\
-^{t}{\bm f}&\bX_{\FF}
\end{pmatrix}\\
&=
\begin{pmatrix}
0&{\bm d}\bG-{\bm f}\\
-^{t}{\bm d}+\bG^{t}{\bm f}&\bX_{\DD}\bG-\bG \bX_{\FF}
\end{pmatrix}.
\end{align*}
Therefore \eqref{eq:FDt} is equivalent to ${\bm d}\bG-{\bm f}={\bm 0}$ and $\bG^{\prime}=\bX_{\DD}\bG-\bG \bX_{\FF}$. It is easy to see that ${\bm d}\bG-{\bm f}={\bm0}$ is equivalent to that $\bG$ is of the form
$\begin{pmatrix}
\epsilon&\o\\
{}^{t}\o& \bar{\bar{G}}
\end{pmatrix}$ and $d_{1}=\epsilon f_{1}$ for some $\epsilon \in \{1,-1\}$.
Then $\bG$ is $O(3)$-valued if and only if
\[
\bar{\bar{G}}=
\begin{pmatrix}
\cos \theta & \kappa \epsilon \sin \theta \\
\sin \theta & -\kappa\epsilon \cos \theta
\end{pmatrix}\]
for some function $\theta$ and some $\kappa \in \{1,-1\}$. 
Under this condition, the equation $\bG^{\prime}=\bX_{\DD}\bG-\bG \bX_{\FF}$ is
\begin{multline*}
\begin{pmatrix}
0&0&0\\
0&-\theta^{\prime}\sin \theta &\kappa\epsilon\theta^{\prime}\cos \theta \\
0&\theta^{\prime}\cos \theta &\kappa \epsilon\theta^{\prime}\sin \theta
\end{pmatrix} = \\ \begin{pmatrix}
0&-\epsilon f_{2} + d_{2}\cos \theta + d_{3}\sin\theta &\kappa \epsilon d_{2}\sin \theta-\kappa\epsilon d_{3}\cos \theta \\
-\epsilon d_{2}+f_{2}\cos \theta  &\kappa\epsilon f_{3}\sin\theta&-f_{3}\cos\theta\\
-\epsilon d_{3}+f_{2}\sin \theta  &-\kappa\epsilon f_{3}\cos \theta&-f_{3}\sin \theta
\end{pmatrix} 
\end{multline*}
which is satisfied if and only if $\theta' = -\kappa\epsilon f_3, \,d_{1}=\epsilon f_{1}, \,d_{2}=\epsilon f_{2}\cos \theta, \,d_{3}=\epsilon f_{2}\sin \theta$, 
The proof of the proposition is concluded.
\end{proof}


We will show that every $2$-regular curve on $\EE^{4}$ admits a frame of type C.

\begin{lemma}\label{lem:2}
Let
\begin{equation}\label{eq:BC}
X_{\BB}= \begin{pmatrix}
0&b_{1}&b_{2}&b_{3}\\
-b_{1}&0&0&0\\
-b_{2}&0&0&0\\
-b_{3}&0&0&0
\end{pmatrix}, \quad\quad 
X_{\CC} = \begin{pmatrix}
0&c_{1}&c_{2}&0\\
-c_{1}&0&0&c_{3}\\
-c_{2}&0&0&0\\
0&-c_{3}&0&0
\end{pmatrix}, 
\end{equation} 
where $b_1, b_2, b_3, c_1, c_2, c_3$ are functions on an interval $I$. Let $\gamma : I \to \EE^{4}$ be a regular curve which has a Bishop frame whose coefficient matrix is $X_{\BB}$. If ${}^{t}\bm{b}$ is nowhere tangent to ${}^{t}(0,1,0)$, then $\gamma$ admits a generalized Bishop frame $\CC$ of type C whose coefficient matrix $X_{\CC}$, where $c_1,c_2,c_3$ are given by 
\begin{equation}\label{eq:cs}
c_{1}=\pm\sqrt{b_{1}^{2}+b_{3}^{2}}, \quad c_{2}=b_{2}, \quad c_{3}=\pm\frac{b_{3}^{\prime}b_{1}-b_{3}b_{1}^{\prime}}{b_{1}^{2}+b_{3}^{2}},
\end{equation}
and the transformation $G = \CC \BB^{-1}$ is of the form 
\begin{equation}\label{eq:bGbc1}
\left(\begin{matrix}
1&0&0&0\\
0&\pm\cos \theta &0&\pm\sin \theta\\
0&0&1&0\\
0&-\sin \theta&0&\cos \theta\\
\end{matrix}\right) 
\end{equation}
for some function $\theta$ such that $\theta'=\pm \frac{b_{3}^{\prime}b_{1}-b_{3}b_{1}^{\prime}}{b_{1}^{2}+b_{3}^{2}}$. 
\end{lemma}

\begin{proof}
By Lemma \ref{lem:1}, it suffices to show that, when \eqref{eq:cs} holds, the equation ${\hG}'=X_{\CC}{\hG} - {\hG}X_{\BB}$ has a solution $\hG : I \to O(4)$ of the form in \eqref{eq:bGbc1}. Let 
\[
{\hG} = \begin{pmatrix} 1 & \o\\ {}^{t}\o & \bG \end{pmatrix}, \quad X_{\BB} = \begin{pmatrix} 0 & {\bm b} \\ -{}^{t}{\bm b} & O \end{pmatrix}, \quad X_{\CC} = \begin{pmatrix} 0 & {\bm c} \\ -{}^{t}{\bm c} & \bX_{\CC} \end{pmatrix}.
\]
Then the equation ${\hG}'=X_{\CC}{\hG} - {\hG}X_{\BB}$ is expressed as 
\begin{equation*}
\begin{pmatrix}
0&\o\\
{}^{t}\o&\bG^{\prime}
\end{pmatrix}
=\begin{pmatrix}
0&{\bm c}\bG-{\bm b}\\
-^{t}{\bm c}+\bG^{t}{\bm b}&\bX_{\CC}\bG
\end{pmatrix},
\end{equation*}
which is equivalent to a pair of equations ${}^{t}{\bm c} = \bG{}^{t}{\bm b}$ and $\bG' = \bX_{\CC} \bG$. 

Assume that \eqref{eq:cs} holds. Since $(b_1, b_3)$ is nowhere vanishing, we have a smooth function $\theta$ such that $(\cos \theta, \sin \theta) = \frac{1}{\sqrt{b_1^2+b_3^2}} (b_1, b_3)$.
Let $\bG$ be
\[
\bG=\begin{pmatrix}
\pm \cos\theta&0&\pm \sin\theta\\
0&1&0\\
-\sin\theta&0&\cos\theta\\
\end{pmatrix}
.\]
By $c_{1}=\pm \sqrt{b_{1}^{2}+b_{3}^{2}}$ and $c_{2}=b_{2}$, we have ${}^{t}{\bm c} = \bG{}^{t}{\bm b}$. By $(\cos \theta, \sin \theta) = \frac{1}{\sqrt{b_1^2+b_3^2}} (b_1, b_3)$
, we  have $\theta'=\frac{b_{3}^{\prime}b_{1}-b_{3}b_{1}^{\prime}}{b_{1}^{2}+b_{3}^{2}}=\pm c_3$．Then it follows that
\[
\bG'=\begin{pmatrix}
\mp\theta^{\prime}\sin\theta&0&\pm\theta^{\prime}\cos\theta\\
0&0&0\\
-\theta^{\prime}\cos\theta&0&-\theta^{\prime}\sin\theta\\
\end{pmatrix}
=
\begin{pmatrix}
-c_{3}\sin\theta&0&c_{3}\cos\theta\\
0&0&0\\
\mp c_{3}\cos\theta&0&\mp c_{3}\sin\theta\\
\end{pmatrix} = \bX_{\CC}\bG,
\]
which concludes the proof of Lemma \ref{lem:2}.


\end{proof}
The following lemma directly follows from Lemma \ref{lem:1} in the case where both $\ZZ_0$ and $\ZZ_1$ are of type B.

\begin{lemma}\label{lem:Bishop}
If a $2$-regular curve $\gamma$ admits a Bishop frame $\BB$ whose coefficient matrix is $X_{\BB}$, then the coefficient matrix of any other Bishop frame is of the form 
\begin{equation}\label{eq:Qc}
\begin{pmatrix}1 & \o \\ {}^{t}\o & Q\end{pmatrix}
X_{\BB}\begin{pmatrix}1 & \o \\ {}^{t}\o & Q^{-1} \end{pmatrix}
\end{equation}
for some constant matrix $Q \in O(3)$. On the contrary, for any $Q \in O(3)$, there exists a Bishop frame whose coefficient matrix is \eqref{eq:Qc}.
\end{lemma}

\begin{theorem}
Every $2$-regular curve $\gamma$ admits a frame of type C.
\end{theorem}

\begin{proof}
Take a Bishop frame $\BB$ of $\gamma$ and let $\begin{pmatrix} 0 & {\bm b} \\ -{}^{t}{\bm b} & O \end{pmatrix}$ be the coefficient matrix. Since $\gamma$ is $2$-regular, $\bm{b}$ is nowhere vanishing. Consider $f : I \to \RR P^{2} ; s \mapsto [{\bm b}(s)]$. By Sard theorem, we have some $\bm{\xi} \in \RR^3 - \{ \o \}$
such that $[\bm{\xi}] \in \RR P^{2} \setminus \operatorname{Image} f$. Take $Q \in O(3)$ so that $\bm{\xi} Q = (0,1,0)$. Then, by Lemma \ref{lem:Bishop}, $\gamma$ admits a Bishop frame whose coefficient matrix is $\begin{pmatrix} 0 & {\bm b}Q  \\ -{}^{t}Q{}^{t}{\bm b} & O \end{pmatrix}$. Since $(0,1,0)$ is never tangent to $\bm{b}Q$, by Lemma \ref{lem:2}, $\gamma$ admits a frame of type C.
\end{proof}

By Bishop's theorem, every regular curve admits a Bishop frame. Therefore, by Proposition \ref{prop:1} and Lemma \ref{lem:2}, we get the following consequence, which implies Theorem \ref{thm:CD}:

\begin{corollary}
If a $2$-regular curve $\gamma$ on $\EE^{4}$ admits a frame of type F, then it admits generalized Bishop frames of all other types B, C and D. In particular, if $\gamma$ admits the Frenet frame, then it admits frames of all other types.  
\end{corollary}

\section{Some examples of curves}\label{sec:ex}



We will present an example of a regular curve which 
does not admit a generalized Bishop frame of type C. Together with Example \ref{ex:noD}, it shows Theorem \ref{thm:noCD}.



\begin{proposition}
Let $\gamma$ be a curve on $\EE^4$ with a Bishop frame $\BB$ whose coefficient matrix is $X_{\BB}=\begin{pmatrix}
0&b_{1}&b_{2}&b_{3}\\
-b_{1}&0&0&0\\
-b_{2}&0&0&0\\
-b_{3}&0&0&0
\end{pmatrix}$, where $b_1, b_2, b_3$ are given by
\begin{align*}
b_{1}(s) & =\begin{cases}
e^{\frac{1}{s}} & s<0 \\
e^{-\frac{1}{s-2}} & 2<s\\
0 & \text{else},
\end{cases} \\
b_{2}(s) & =\begin{cases}
e^{-\frac{1}{s(-s+1)}} & 0<s<1 \\
0 & \text{else},
\end{cases} \\
b_{3}(s) & =\begin{cases}
e^{-\frac{1}{(s-1)(-s+2)}} & 1<s<2 \\
0 & \text{else}.
\end{cases}
\end{align*}
Then $\gamma$ is a regular curve which does not admit a frame of type C.
\end{proposition}

\begin{proof}
Assume that $\gamma$ admits a generalized Bishop frame $\CC$ of type C whose coefficient matrix is $\begin{pmatrix}
0&c_{1}&c_{2}&0\\
-c_{1}&0&0&c_{3}\\
-c_{2}&0&0&0\\
0&-c_{3}&0&0
\end{pmatrix}$. 
By Lemma \ref{lem:1}, the transformation $G=\CC \BB^{-1}$ satisfies the differential equation $G^{\prime}=X_{\CC} G - G X_{\BB}$. Let
\[
{\hG} = \begin{pmatrix} 1 & \o\\ {}^{t}\o & \bG \end{pmatrix}, \quad X_{\BB} = \begin{pmatrix} 0 & {\bm b} \\ -{}^{t}{\bm b} & O \end{pmatrix}, \quad X_{\CC} = \begin{pmatrix} 0 & {\bm c} \\ -{}^{t}{\bm c} & \bX_{\CC} \end{pmatrix}.
\]
Then $G^{\prime}=X_{\CC} G - G X_{\BB}$ is equivalent to ${}^{t}\bm{c}=\bG{}^{t}\bm{b}$ and $\bG^{\prime}=\bX_{\CC}\bG$. By $\bG^{\prime}=\bX_{\CC}\bG$, we have 
\[
\bG = \begin{pmatrix}
\cos\theta&0&\sin\theta\\
0&1&0\\
-\sin\theta&0&\cos\theta\\
\end{pmatrix} \overline{G}_0
\]
for a function $\theta$ and a constant matrix $\overline{G}_0 \in O(3)$. Take $\bm{\tilde{b}}$ so that  
${}^{t}\bm{\tilde{b}}= \overline{G}_0 {}^{t}\bm{b}$ and $\tilde{b}_i$ be the $i$-th component of $\bm{\tilde{b}}$ $(i=1,2,3)$.
Then, by ${}^{t}\bm{c}=\bG{}^{t}\bm{b}$, we have 
\begin{equation}\label{eq:cbt}
\left(\begin{matrix}c_{1}\\
c_{2}\\
0
\end{matrix}
\right)
=\left(
\begin{matrix}
\cos\theta&0&\sin\theta\\
0&1&0\\
-\sin\theta&0&\cos\theta\\
\end{matrix}
\right)\left(\begin{matrix}\tilde{b}_{1}\\
\tilde{b}_{2}\\
\tilde{b}_{3}
\end{matrix}
\right).\end{equation}

Let $e_1={}^{t}(1, 0, 0)$, $e_{2}={}^{t}(0, 1, 0)$, $e_3={}^{t}(0, 0, 1)$. Then, ${}^{t}\bm{b}$ is tangent to $e_1$ for $s<0$, tangent to $e_2$ for $0< s < 1$ and tangent to $e_3$ for $s>1$.

First let us show a contradiction assuming that ${}^{t}\bm{\tilde{b}}$ is nowhere tangent to $e_2$. 
In this case, since $(\tilde{b}_1, \tilde{b}_3)$ is nowhere vanishing, we have $(\cos \theta, \sin \theta) = \frac{1}{\sqrt{\widetilde{b}_{1}^{2}+\widetilde{b}_{3}^{2}}} (\widetilde{b}_{1}, \widetilde{b}_{3})$. Since $\bm{b}$ is tangent to $e_1$ for $s < 0$, it is clear that $\theta$ is constant for $s < 0$. Similarly we have that $\theta$ is constant also for $0 < s< 1$ and for $1 < s$. Therefore $G$ is constant everywhere. By \eqref{eq:cbt}, it follows that $\overline{G} {}^{t}\bm{b}$ is contained in the subspace $\RR^{2} \oplus 0 \subset \RR^{3}$
for a constant orthogonal matrix $G$. But this is not possible, since the image of ${}^{t}\bm{b}$ is clearly not contained in any $2$-dimensional subspace of $\RR^{3}$. Thus we have a contradiction.

Now let us consider the case where $\overline{G}_0 e_1$ is tangent to $e_2$. In this case, ${}^{t}\bm{\tilde{b}}$ is tangent to $e_2$ for $s<0$ and $s>2$, and not tangent to $e_2$ for $0<s<1$ and $1<s<2$.
It follows that $(\tilde{b}_1, \tilde{b}_3)$ is nowhere vanishing for $0<s<1$ and $1<s<2$. As in the last paragraph, it follows that $\theta$ is constant for $0<s<2$.
Since $\overline{G}_0 e_1$ is tangent to $e_2$, $\bG_{0}$ is of the form $\begin{pmatrix}0&g_{23}&g_{24}\\ \pm 1&0&0\\0&g_{43}&g_{44}\end{pmatrix}$. Then ${}^{t}\bm{\tilde{b}}=e^{-\frac{1}{s(1-s)}}\begin{pmatrix} g_{23} \\
0\\
g_{43}
\end{pmatrix}$ for $0<s<1$ and ${}^{t}\bm{\tilde{b}}=e^{-\frac{1}{(s-1)(-s+2)}}\begin{pmatrix} g_{24} \\
0\\
g_{44}
\end{pmatrix}$ for $1<s<2$. By \eqref{eq:cbt}, it follows that $c_2$ is identically zero for $0 < s < 2$. Then the image of $\overline{G} {}^{t}\bm{b}$ for $0<s<2$ is contained in the subspace $\RR \oplus 0 \oplus 0 \subset \RR^{3}$.
On the other hand, ${}^t\bm{b}$ is tangent to $e_2$ for $0<s<1$ and ${}^t\bm{b}$ is tangent to $e_3$ for $1<s<2$. Therefore it is a contradiction.


We omit the detail of the proof of the remaining two cases; the case where $\bG_{0} e_2$ is tangent to $e_2$, and the second case where  $\bG_{0} e_{3}$ is tangent to $e_{2}$. In both cases, we can show a contradiction by an argument similar to the last two cases.
Therefore, we conclude that $\gamma$ does not admit a generalized Bishop frame of type C.

\end{proof}

Note that a curve with a Bishop frame with the following curvature, which is simpler than the last example, admits a frame of type C:
\begin{align*}
b_{1}(s) & =\begin{cases}
e^{\frac{1}{s}} & s<0 \\
0 & \text{else},
\end{cases} \\
b_{2}(s) & =\begin{cases}
e^{-\frac{1}{s(-s+1)}} & 0<s<1 \\
0 & \text{else},
\end{cases} \\
b_{3}(s) & =\begin{cases}
e^{-\frac{1}{(s-1)(-s+2)}} & 1<s<2 \\
0 & \text{else}.
\end{cases}
\end{align*}
Indeed, we can take a transformation $G$ from the Bishop frame to a frame of type C so that $\overline{G}_0 e_2$ is tangent to $e_2$.





\end{document}